\documentclass[12pt,bezier]{article}
\usepackage{amssymb}
\usepackage{mathrsfs}
\usepackage{amsmath}
\usepackage{amsfonts,amsthm,amssymb}
\usepackage{amsfonts}
\usepackage{graphics}
\newtheorem{lem}{Lemma}
\newtheorem{thm}[lem]{Theorem}
\newtheorem{cor}[lem]{Corollary}

\newtheorem{rem}[lem]{Remark}
\newtheorem{defi}[lem]{Definition}

\begin{document}

\title{Component edge connectivity of the folded hypercube
}
\author{ Shuli Zhao, \quad Weihua Yang \footnote{Corresponding author. E-mail: ywh222@163.com,~yangweihua@tyut.edu.cn}\\
\\ \small Department of Mathematics, Taiyuan University of Technology,\\
\small  Taiyuan Shanxi-030024,
China}
\date{}
\maketitle

{\small{\bf Abstract.}\quad  The $g$-component edge connectivity $c\lambda_g(G)$ of a non-complete graph $G$ is the minimum number of edges whose deletion results in a graph with at least $g$ components. In this paper, we determine the component edge connectivity of the folded hypercube $c\lambda_{g+1}(FQ_{n})=(n+1)g-(\sum\limits_{i=0}^{s}t_i2^{t_i-1}+\sum\limits_{i=0}^{s} i\cdot
2^{t_i})$ for $g\leq 2^{[\frac{n+1}2]}$ and $n\geq 5$, where $g$ be a positive integer and $g=\sum\limits_{i=0}^{s}2^{t_i}$ be the
decomposition of $g$ such that $t_0=[\log_{2}{g}],$ and
$t_i=[\log_2({g-\sum\limits_{r=0}^{i-1}2^{t_r}})]$ for $i\geq 1$.


\section{Introduction}
Let $G$ be a non-complete graph. A $g$-component edge cut of $G$ is a set of edges whose deletion results
in a graph with at least $g$ components. The  $g$-component edge connectivity $c\lambda_{g}(G)$ of a graph $G$ is the size of the smallest  $g$-component edge cut of $G$. By the definition of the $c\lambda_{g}(G)$, it can be seen that $ c\lambda_{g+1}(G)\geq c\lambda_{g}(G)$ for every positive integer $g$.

An interconnection network is usually modeled by a connected graph in which vertices represent processors and edges represent links between processors. The usual edge connectivity  $\lambda(G)$ of  a graph $G$ is the minimum number of edges whose deletion results in a disconnected graph. The edge connectivity is one of the important parameters to evaluate the reliability and fault tolerance of a network. The $g$-component edge connectivity is an extension of the usual edge connectivity $c\lambda_{2}(G)$. The $g$-component connectivity and $g$-component edge connectivity were introduced in \cite{chartrand} and \cite{sampathkumar} independently. In \cite{hsu,zhao}, Hsu et al. and Zhao et al. determined the $g$-component connectivity of the hypercube $Q_{n}$ for $2\leq g\leq n+1$ and $ n+2 \leq g\leq 2n-4$ respectively. As an invariant of the hypercube, the folded hypercube was first proposed by El-Amawy and Latifi \cite{El-Amawy}, is one of the most potential interconnection networks. There are some results about the folded hypercubes \cite{zhang, zhu, zhu2, lai, xu}, so in this paper, we determine the $g$-component edge connectivity of the folded hypercube $FQ_{n}$ for $g\leq 2^{[\frac{n+1}2]}, n\geq 5$.

The $n$-dimensional hypercube $Q_{n}$ is an undirected graph $Q_{n}=(V,E)$ with $|V|=2^{n}$ and $|G|=n\cdot 2^{n-1}.$ Each vertex can be represented by an $n$-bit binary string. There is an edge between two vertices whenever there binary string representation differs in only one bit position. The folded hypercube is an enhancement of the hypercube $Q_{n}$ and $FQ_{n}$ is obtained by adding a perfect matching $M$ on the hypercube, where $M=\{(u,\overline{u})|u\in V(Q_{n})\}$ and $\overline{u}$ represents the complement of the vertex $u,$ that is, all their binary strings are complement and $\overline{0}=1$ and $\overline{1}=0$. One can seen that $E(FQ_{n})=E(Q_{n})\bigcup M.$ For convenience, $FQ_{n}$ can be expressed as $D_{0}\bigotimes D_{1},$ where $D_{0}$ and $D_{1}$ are $(n-1)$-dimensional subcubes induced by the vertices with the $i$-th coordinate $0$ and $1$ respectively. The
3-dimensional and 4-dimensional folded hypercubes are shown in the
following Figure 1 and Figure 2,  respectively.
\begin{center}
\scalebox{0.25}{\includegraphics{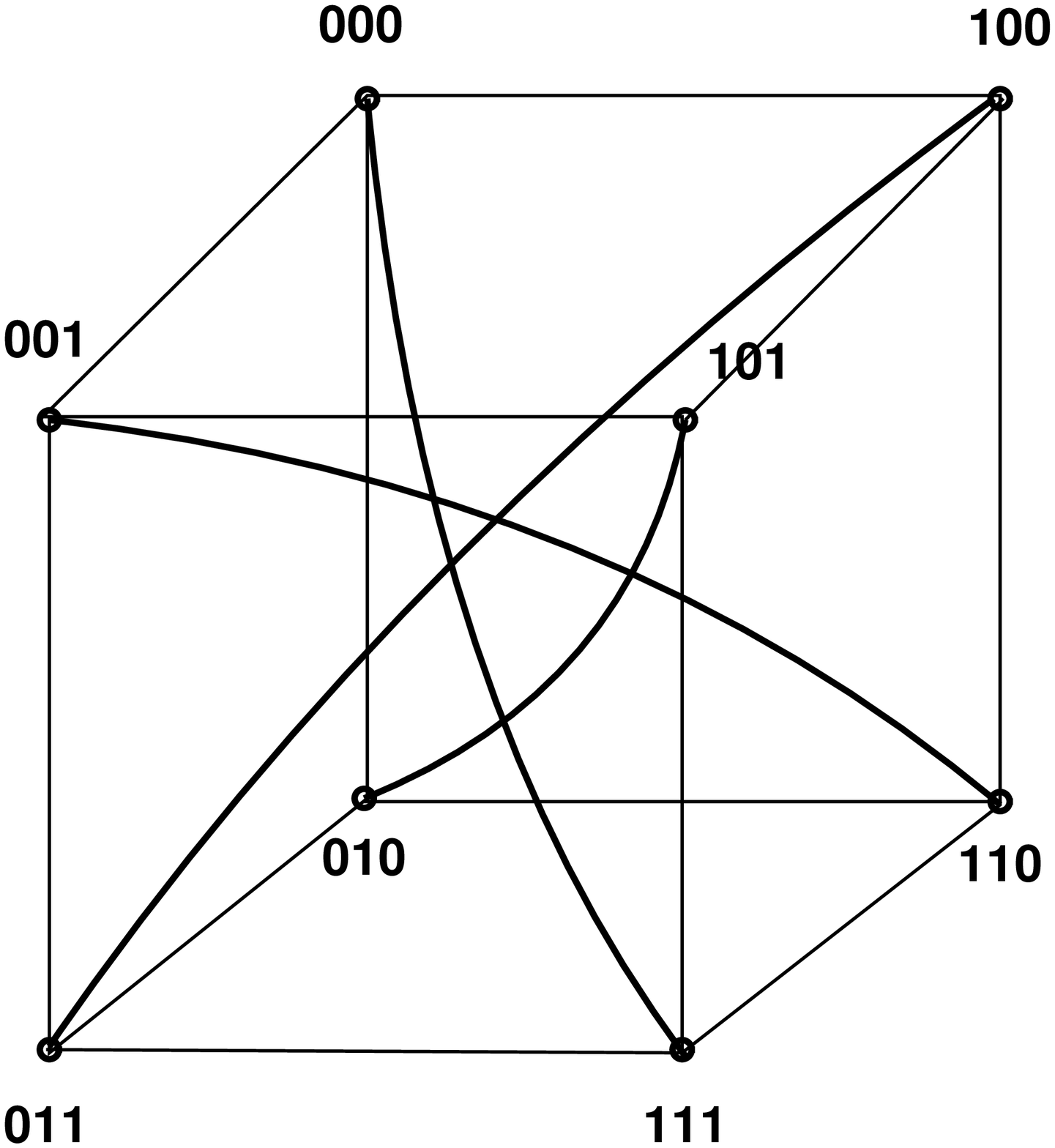}}\\
Figure 1. The 3-dimensional Folded hypercube.
\end{center}

\begin{center}
\scalebox{0.5}{\includegraphics{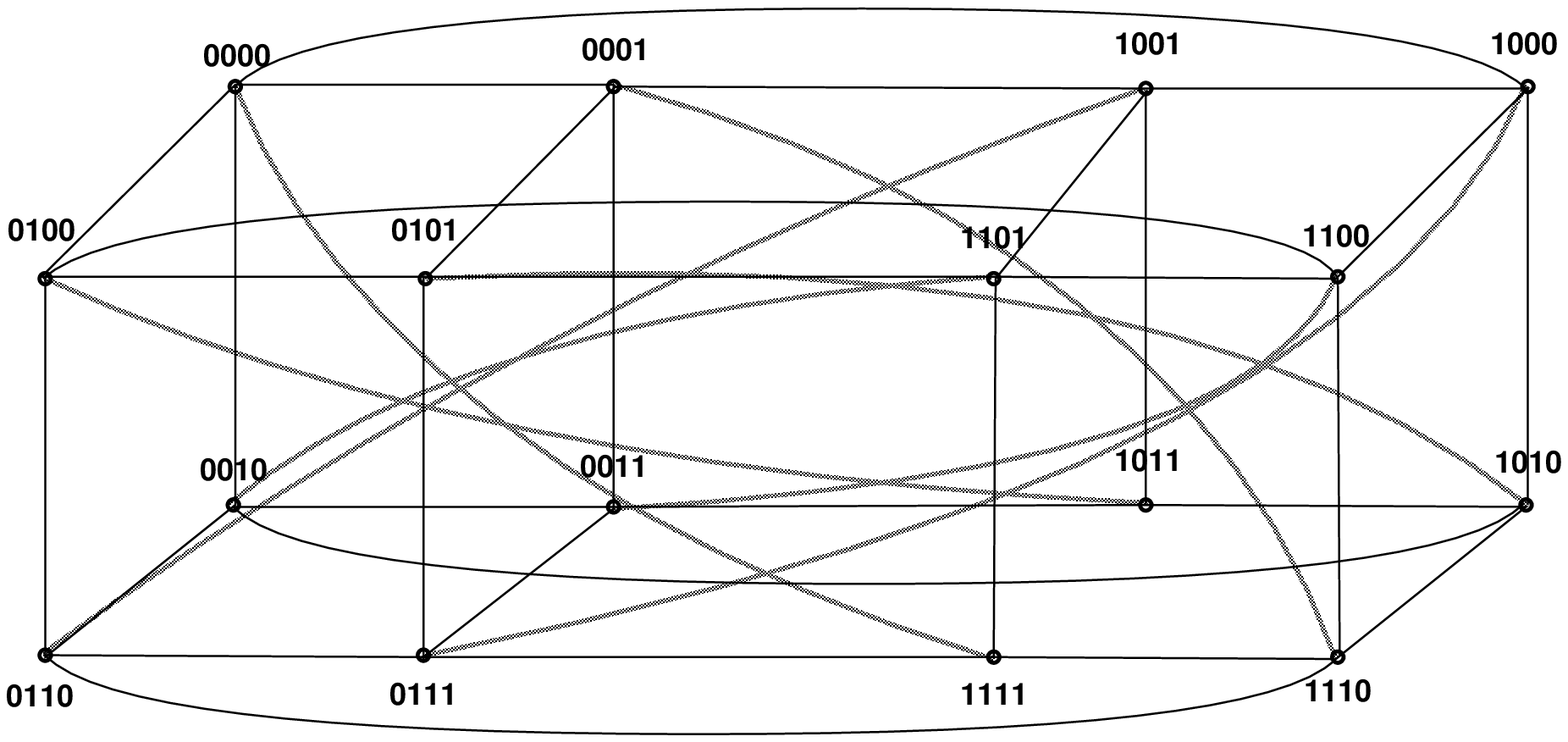}}\\
Figure 2. The 4-dimensional Folded hypercube.
\end{center}

\section{Preliminaries}

Let $m$ be an integer and $m=\sum\limits_{i=0}^{s}2^{t_i}$ be the
decomposition of $m$ such that $t_0=[\log_{2}{m}]$ and
$t_i=[\log_2({m-\sum\limits_{r=0}^{i-1}2^{t_r}})]$ for $i\geq 1$.
Let $X$ be a vertex set of a graph $G$ with $|X|=m$. We denote by
$\frac{ex_m}{2}$ the maximum number of edges of the subgraph of a graph $G$ induced
by $m$ vertices, i.e., $ex_m=\max\{2|E(G[X])|:~ X\subset V(G)
\mbox{ and }|X|=m\}$, is the maximum sum of degrees of
the vertices in the subgraph of a graph $G$ induced by $m$ vertices. In \cite{li}, the authors determined the $ex_m(Q_n)$, which play an important role in this proof.

\begin{thm}\label{thm1}$(\cite{li})$
Let $X$ be a vertex set of $Q_n$ with size $m$. Then
$ex_m(Q_n)=\sum\limits_{i=0}^{s}t_i2^{t_i}+\sum\limits_{i=0}^{s}2\cdot
i\cdot 2^{t_i}$.
\end{thm}

\begin{lem}$(\cite{li})$\label{lem3}
If $m_0\leq m_1$. Then  $ex_{m_0+m_1}(Q_n)\geq ex_{m_0}(Q_n)+ex_{m_1}(Q_n)+2m_0$.
\end{lem}

Before our main results, we need several results focused on the induced subgraph of $FQ_n$ with maximum number of edges \cite{boals,chen,katseff,li,ra}

\begin{defi}$(\cite{boals})$\label{lem3}
 A set of $m$ vertices of $FQ_n$ $(Q_n)$ is said to be a composite set for $m$ if the number of edges of the subgraph induced by these $m$ vertices is not less than the number of edges of subgraphs induced by any other set of $m$ vertices of $FQ_n$ $(Q_n)$. A composite folded hypercube of $FQ_n$ (composite hypercube of $Q_n$) is defined to be a subgraph of $FQ_n$ $(Q_n)$, which is induced by some composite set of $FQ_n$ $(Q_n)$.
\end{defi}

For convenience, the vertex $u=u_{n}u_{n-1}\cdots u_{1}$ of $n$-dimensional folded hypercube also can be represented by decimal number $\sum_{i=1}^{n}u_{i}2^{i-1}$ in this paper.

\begin{defi}$(\cite{katseff,ra})$\label{lem3}
The subgraph induced by vertex set $\{0,1,\cdots ,m-1\}$ (under decimal representation) of $FQ_n$ $(Q_n)$, denoted by $LF_{m}$ $(L_{m})$, is called as an incomplete folded hypercube (incomplete hypercube) on $m$ vertices of $FQ_n$ $(Q_n)$.
\end{defi}

\begin{defi}$(\cite{ra})$\label{lem3}
A reverse incomplete folded hypercube (reverse incomplete hypercube) on $m$ vertices of $FQ_n$ $(Q_n)$ is the subgraph induced by $\{2^{n}-1, 2^{n}-2 \cdots ,2^{n}-m\}$ and is denoted by $RF_{m}$ $(R{m})$, for $1\leq m \leq 2^{n}$.
\end{defi}

\begin{lem}$(\cite{ra})$\label{lem3}
$(L_{m})$ is isomorphic to $R_{m}$ and $(LF_{m})$ is isomorphic to $RF_{m}$, for $1\leq m \leq 2^{n}$.
\end{lem}

\begin{lem}$(\cite{chen,katseff,ra})$\label{lem3}
For $1\leq m \leq 2^{n}$, both $V(LF_{m})$ and $V(RF_{m})$ $((V(L_{m})$ and $V(R_{m}))$ are composite sets of $FQ_n$ $(Q_n)$.
\end{lem}

\begin{lem}\label{lem3}
\[ex_{m}(FQ_{n})=\left\{
\begin{array}{ll}
          \sum\limits_{i=0}^{s}t_i2^{t_i}+\sum\limits_{i=0}^{s}2\cdot i\cdot 2^{t_i},\hspace{3cm} 1\leq m \leq 2^{n-1}\\
\sum\limits_{i=0}^{s}t_i2^{t_i}+\sum\limits_{i=0}^{s}2\cdot i\cdot 2^{t_i}+m-2^{n-1}, \hspace{1cm} 2^{n-1}\leq m \leq 2^{n} \\
\end{array}
\right.\]
\end{lem}
\begin{proof} By Theorem $1$ and Lemma $7$, the results holds.
\end{proof}

For $m \leq 2^{n-1}$, we note that $ex_{m}(FQ_{n})= ex_{m}(Q_{n}).$

 Denote by $E_{X}$ the set of edges in which
each edge contains exactly one end vertex in $X$.
 By the lemma above,  the following lemma holds.
\begin{lem}\label{lem2}
Let $X$ be a subset of $V(FQ_n)$ with $|X|=m$, where $m$ be an positive integer and $m=\sum\limits_{i=0}^{s}2^{t_i}$. Then $E_{X}$ contains
at least $(n+1)|X|-ex_m$ edges. Moreover,  the function $\xi(m)=(n+1)|X|-\frac{ex_m}2$
is strictly increasing (respect to $m$) for $m\leq2^{[\frac{n+1}2]}$.
\end{lem}
\begin{proof} By Lemma $8$, the first part is clearly.  As $m\leq2^{[\frac{n+1}2]}$ and $m=\sum\limits_{i=0}^{s}2^{t_i}$, then
$\xi(m+1)-\xi(m)=(n+1)-(s+1)$, which implies $\xi(m)$ is strictly increasing (respect to $m$) for $m\leq2^{[\frac{n+1}2]}$.
\end{proof}

\begin{lem}\label{lem2}
Let $X$ be a subset of $V(FQ_n)$ with $|X|=m$, if $m \leq 2 ^{n-1}$, then $(n-1)m- ex_{m}(FQ_n)\geq 0$.
\end{lem}
\begin{proof} Let $FQ_n =D_{0}\bigotimes D_{1},$ as $m \leq 2 ^{n-1}$, by Definition $4$, Lemma $7$ and $8$, we can take a subgraph $G_{0}$ of a $(n-1)$-dimensional subcube $D_{i}$ (i=0, 1) such that $2|E(G_{0})|=ex_{m}$. Thus, if $m \leq 2 ^{n-1}$, then $(n-1)m- ex_{m}(FQ_n)\geq 0$.
\end{proof}
\begin{lem}\label{lem2}
If $\sum_{i=1}^{r}ex_{m_i}=m$ and $m_i>0$, then  $\sum_{i=1}^{r}ex_{m_i}\leq ex_{m-r+1}$.
\end{lem}
\begin{proof}
To prove the result, we just need to prove that when $k\leq l$, $ex_{k+1}+ex_l\leq ex_{k+l}$.

Note that $ex_{k+1}=ex_k+2(s+1)$, where $k=2^{t_0}+\cdots+2^{t_s}$. By Lemma $2$, $ex_{k+l}\geq ex_k+ex_l+2k=ex_{k+1}+ex_l+2k-2(s+1)\geq ex_{k+1}+ex_l$. By using the inequality, it can be seen that $ex_{m_1}+ex_{m_2}+\cdots + ex_{m_r}\leq ex_{m_1+m_2-1}+ex_{m_3}+\cdots+ex_{m_r}\leq \cdots\leq ex_{m-r+1}$.

\end{proof}

\section{Main Results}

In this section, we determine the $c\lambda_g(FQ_n)$.

\begin{thm}\label{thm2}
 $c\lambda_{g+1}(FQ_n)=(n+1)g-\frac{ex_g}2$ for $g\leq 2^{[\frac{n+1}2]}, n\geq 5$.
\end{thm}

\begin{proof}For convenience sake, we assume that $n$ is odd.

First, we show that $c\lambda_{g+1}(FQ_n)\leq (n+1)r-\frac{ex_g}2$ by constructing a $(g+1)$-component edge cut $F$ with size $(n+1)g-\frac{ex_g}2$.
Let $G_1$ be the subgraph induced by the vertex set $V=\{0,1,\cdots ,g-1\}$,
where $g\leq 2^{n-1}$ and $g=\sum\limits_{i=0}^{s}2^{t_i}$, then by Definition $4$ Lemma $7$ and $8$, $|E(G_1)|=\sum\limits_{i=0}^{s}t_i2^{t_i}+\sum\limits_{i=0}^{s}2\cdot
i\cdot 2^{t_i}= ex_{g}(FQ_n).$ Let $F=E_{V(G_1)}\cup E(G_1)$, then $F$ is a $(g+1)$-component edge cut with size $(n+1)g-\frac{ex_g}2$.

Next, we show that $c\lambda_{g+1}(FQ_n)\geq (n+1)g-\frac{ex_g}2$. Let $F$ be a $c\lambda_{g+1}$-cut, then we have that $FQ_n-F$ has exactly $g+1$ components. We denote by $C_1,C_2,\cdots,C_{g+1}$ the $g+1$ components in $FQ_n-F$, and $C_{g+1}$ be the largest one. We will prove the result by the following cases:

{Case 1.} $|V(C_{g+1})|<2^{n-2}$

If $|V(C_{g+1})|<2^{n-2}$, then there exists a partition of $\{C_1,C_2,\cdots,C_{g+1}\}$ such that $|\cup_{j=1}^{k}V(G_{i_j})|\geq 2^{n-2}$ and $|\cup_{j=k+1}^{g+1}V(G_{i_j})|\geq 2^{n-2}$. So we just need to show that $|\bigcup_{i=1}^{g+1}E_{V(C_{i})}(FQ_{n})|\geq (n+1)g-\frac{ex_g}2.$

Let $X\subset V(FQ_n)$ and $\overline{X}=V(FQ_n)\setminus X$. If $2^{n-2}\leq |X|\leq |\overline{X}|$ and let $|X|=m$,
then denote by
$m'=m-2^{t_0}$. Clearly, $m'< 2^{n-2}$.  By Lemma $10$, we have the following:
\begin{equation*} \label{eq:1}
\begin{split}
(n+1)m-ex_m& =
(n+1)2^{t_0}+(n+1)(2^{t_1}+\cdots+2^{t_s})-t_02^{t_0-1}-(\sum\limits_{i=1}^{s}t_i2^{t_i}+\sum\limits_{i=1}^{s}2\cdot
i\cdot
2^{t_i})\\
 &\geq (n+1- t_0)2^{t_0}+(n+1)\sum\limits_{i=1}^s2^{t_i}-(\sum\limits_{i=1}^{s}t_i2^{t_i}+\sum\limits_{i=1}^{s}2\cdot
i\cdot 2^{t_i}))\\
&=(n+1- t_0)2^{t_0}+(n+1)r'-ex_{r'}-2r'\\
&=(n+1- t_0)2^{t_0}+(n-1)r'-ex_{r'}\\
&\ge (n+1- t_0)2^{t_0}\\
&=(n+1-t_0)\cdot2^{(t_0-\frac{n+1}2)}\cdot2^{\frac{n+1}2}
 \end{split}
 \end{equation*}
Note that $n-2\leq t_0\leq n-1$, then we have
$(n+1-t_0)\cdot2^{(t_0-\frac{n+1}2)}\geq \frac{3(n+1)}4$ for $n\geq 7$. So one can see that $|F|\geq (n+1)\cdot2^{\frac{n+1}2}-\frac{n+1}{4}\times 2^{\frac{n+1}2}=\frac{3(n+1)}{4}\times 2^{\frac{n+1}2} \geq (n+1)g-\frac{ex_g}2 $ when $|V(C_{g+1})|<2^{n-2}$ .

{Case 2.} $|V(C_{g+1})|\geq 2^{n-2}$

Let $|V(C_i)|=m_i$ and $m=\sum_{i=1}^{g}m_i$. Clearly, we may assume $m< 2^{n-2}$.

 If $m=g$, then $m_i=1$ and thus $|F|\geq (n+1)g-\frac{ex_g}2$. Therefore, we may assume $m>g$. Let $S=\cup_{i=1}^gV(C_i)$ and $F'=F\cap E(G[S])$.
Note that $E_{V(C_i)}\subset F$, then we have $|F|\geq |\cup_{i=1}^g E_{V(C_i)}| = |E_{V(C_1)}|+\cdots + |E_{V(C_g)}|-|F'|$. Since $|E_{V(C_i)}|=m_i\cdot (n+1)-2|E(C_i)|$ and $|F'|\leq \frac{ex_m}2-|\cup_{i=1}^g E(C_i)|$,  we have the following.

\begin{equation*} \label{eq:1}
\begin{split}
|F|&\geq |\cup_{i=1}^g E_{V(C_i)}| = |E_{V(C_1)}|+\cdots + |E_{V(C_g)}|-|F'|\\
&\geq \sum_{i=1}^{g}(m_i\cdot (n+1)-2|E(C_i)|)-\frac{ex_m}2 + |\cup_{i=1}^g E(C_i)|
\\
 &= m(n+1)-2\sum_{i=1}^{g}|E(C_i)|-\frac{ex_m}2 + \sum_{i=1}^{g}|E(C_i)|\\
&=m(n+1)-\frac{ex_m}2-\sum_{i=1}^{g}|E(C_i)|\\
&\geq m(n+1)-\frac{ex_m}2-\sum_{i=1}^{g}\frac{ex_{m_i}}2
 \end{split}
 \end{equation*}

By Lemma $11$, we have  $|F|\geq m(n+1)-\frac{ex_m}2-\sum_{i=1}^{g}\frac{ex_{m_i}}2\geq m(n+1)-\frac{ex_m}2-\frac{ex_{m-g+1}}2$. We next show that  $$m(n+1)-\frac{ex_m}2-\frac{ex_{m-g+1}}2\geq (n+1)g-\frac{ex_g}2~~~~~~~~~~~~~~~~.$$

For any $g\leq 2^{\frac{(n+1)}2}$ and $S=\{v_1,v_2,\cdots,v_g\}$, $|\cup_{i=1}^g E_{v_i}|\geq (n+1)g-\frac{ex_g}2$ holds as $|\cup_{i=1}^g E_{v_i}|\geq (n+1)g-2|E(FQ_n[S])|+|E(FQ_n[S])|=(n+1)g-|E(FQ_n[S])|\geq (n+1)g-\frac{ex_g}2$.

Let $FQ_n=D_0 \bigotimes D_1$. Since $m<2^{n-2}$, we may pick a subgraph $G_1$ of $m$ vertices in $D_0$, where $G_1$ is the subgraph of $FQ_n$ induced by $\{0,1,2,\cdots, m-1\}$. So $|E(G_1)|=\frac{ex_m}2$. Since $m-g+1<m$, we may pick a subgraph $G_2\subset G_1$ such that $|E(G_2)|=\frac{ex_{m-g+1}}2$ (here we pick the subgraph $G_2$ has the same structural property as $G_1$). Label the vertices in $G_2$ by $v_m,v_{m-1},\cdots,v_{g+1},v_{g}$ such that $d_{G_2}(v_g)=s+1$, where $m-g=2^{t_0}+2^{t_1}+\cdots+2^{t_s}$, and label the vertices in $V(G_1-G_2)$ by $v_1,v_2,\cdots,v_{g-1}$. See Figure 1 (a).

\begin{center}
\scalebox{0.9}{\includegraphics{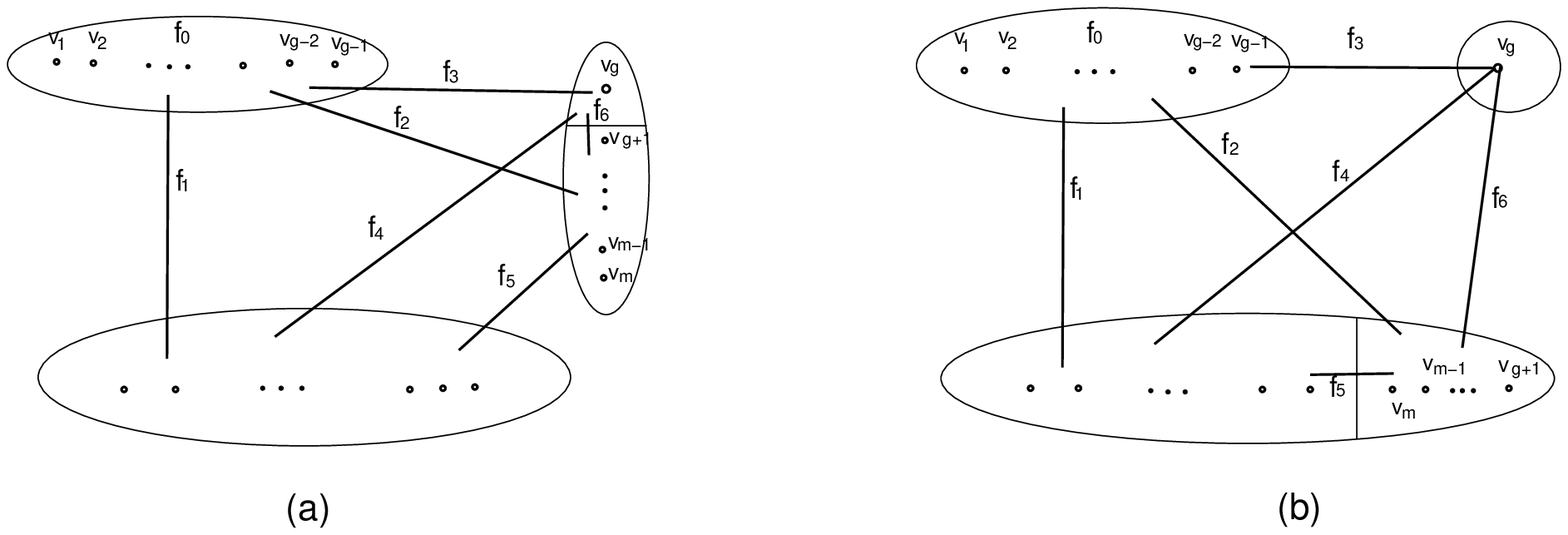}}
Figure 1. The edges between the components.
\end{center}

Let $S=\{v_1,\cdots,v_{g-1}\}$ and $f_0=|E(FQ_n[S])|$. Clearly, $f_0+f_1+f_2+f_3+f_4+f_5=m(n+1)-2|E(G_1)|+(|E(G_1)|-|E(G_2)|) =m(n+1)-\frac{ex_m}2-\frac{ex_{m-g+1}}2$. See Figure 1 (b). By Claim 5, $f_0+f_1+f_2+f_3+f_4+f_6\geq (n+1)g-\frac{ex_g}2$. Thus if $f_5\geq f_6$, then $m(n+1)-\frac{ex_m}2-\frac{ex_{m-g+1}}2=f_0+f_1+f_2+f_3+f_4+f_5\geq f_0+f_1+f_2+f_3+f_4+f_6\geq  (n+1)g-\frac{ex_g}2$. Note that $G_1\subset D_0$, then there  $f_5\geq m-g\geq s+1=f_6$

We omit the argument for even $n$.

Thus, $|F|\geq (n+1)g-\frac{ex_g}2$ and then $c\lambda_{g+1}(FQ_n)=(n+1)g-\frac{ex_g}2$.
\end{proof}

\begin{cor}
Let $F$ be a $c\lambda_{g+1}$-cut of the hypercube $FQ_n$. Then $FQ_n-F$ contains $g$ isolated vertices for $g\leq 2^{[\frac{(n+1)}2]}, n\geq 5$.
\end{cor}
\begin{proof}
The result holds easily.

\end{proof}

\section{Conclusion}

In this paper, we studied the component edge connectivity of the folded hypercube. The component (edge) connectivity is a generalization of standard (edge) connectivity of graphs, see \cite{chartrand,sampathkumar},  which can be viewed as a  measure of robustness of interconnection networks. The standard connectivity of folded hypercubes (or classic networks) have been studied by many authors, but there are few papers on the component (edge) connectivity of networks.  This paper introduced an idea to consider the  $g$-component edge connectivity of the folded hypercube (or cube-based networks), but our result is not complete. The problem is still open for $g> 2^{[\frac{(n+1)}2]}+1$.

\section{Acknowledgements}


The research is supported by NSFC (No.11671296, 11301217, 11301085), SRF for ROCS, SEM and Natural Sciences
Foundation of Shanxi Province (No. 2014021010-2), Fund Program for the Scientific Activities of Selected
Returned Overseas Professionals in Shanxi Province

\begin{rem}
The work was included in the MS thesis of the first author in [On the component connectiviy of hypercubes and folded hypercubes, MS Thesis at Taiyuan University of
Technology, 2017].
\end{rem}

\end{document}